\documentclass[12pt]{amsart}
\usepackage{amssymb, amsmath, amsfonts, cite}
\usepackage[colorlinks]{hyperref}
\usepackage[capitalize]{cleveref}
\usepackage[a4paper]{geometry}

\newtheorem{thm}{Theorem}[section]
\newtheorem{prob}[thm]{Problem}

\usepackage{enumitem}
\setlist[enumerate,itemize]{topsep=4pt, parsep=3pt, itemsep=3pt, 
leftmargin=0.5cm, 
listparindent=20pt, labelsep=5pt, itemindent=0pt, partopsep=5pt, label=\roman*)}

\makeatletter \@addtoreset{equation}{section} \makeatother

\def\={\;=\;}
\def\lle{\;\le\;}

\def\Z{\mathbb{Z}}
\def\N{\mathbb{N}}

\crefname{case}{Case}{Cases}
\crefname{case}{Case}{Cases}
\creflabelformat{case}{~\upshape#2#1#3}
\crefname{casethm}{Case}{Cases}
\crefname{casethm}{Case}{Cases}
\creflabelformat{casethm}{~\upshape#2#1#3}
\crefname{cond}{Condition}{Conditions}
\Crefname{cond}{Condition}{Conditions}
\creflabelformat{cond}{~\upshape(#2#1#3)}
\crefname{def}{Definition}{Definitions}
\Crefname{def}{Definition}{Definitions}
\creflabelformat{def}{~\upshape(#2#1#3)}
\crefname{ineq}{Ineq.}{Ineqs.}
\Crefname{ineq}{Inequality}{Inequalities}
\creflabelformat{ineq}{~\upshape(#2#1#3)}
\crefname{prob}{Problem}{Problems}
\Crefname{prob}{Problem}{Problems}
\creflabelformat{prob}{~\upshape#2#1#3}
\crefname{rl}{Relation}{Relations}
\crefname{rl}{Relation}{Relations}
\creflabelformat{rl}{~\upshape(#2#1#3)}
\crefname{thm}{Theorem}{Theorems}
\Crefname{thm}{Theorem}{Theorems}
\creflabelformat{thm}{~\upshape#2#1#3}

\crefrangeformat{equation}{Eqs.\!~(#3#1#4) ---~(#5#2#6)}

\allowdisplaybreaks

\usepackage{xcolor}

\begin{document}

\title[]{The Tutte's condition in terms of graph factors}

\author[H. Lu]{Hongliang Lu}
\address{School of Mathematics and Statistics, Xi'an Jiaotong university, 710049 Xi'an, P.\ R.\ China}
\email{luhongliang@mail.xjtu.edu.cn}
\thanks{Lu is supported by the National Natural Science Foundation of China, No.\ 11471257.}

\author[D.G.L. Wang]{David G.L. Wang$^{\dag\ddag}$}
\address{
$^\dag$School of Mathematics and Statistics, Beijing Institute of Technology, 102488 Beijing, P.\ R.\ China\\
$^\ddag$Beijing Key Laboratory on MCAACI, Beijing Institute of Technology, 102488 Beijing, P.\ R.\ China}
\email{glw@bit.edu.cn}
\thanks{Wang is supported by the National Natural Science Foundation of China, No.\ 11671037}

%05C75 Structural characterization of families of graphs

\begin{abstract}
Let $G$ be a connected general graph of even order, with a function $f\colon V(G)\to\Z^+$.
We obtain that $G$ satisfies the Tutte's condition
\[
o(G-S)\le \sum_{v\in S}f(v)\qquad\text{for any nonempty set $S\subset V(G)$},
\]
with respect to $f$ 
if and only if $G$ contains an $H$-factor for any function $H\colon V(G)\to 2^\N$ such that 
$H(v)\in \{J_f(v),\,J_f^+(v)\}$ for each $v\in V(G)$, where 
the set $J_f(v)$ consists of the integer $f(v)$
and all positive odd integers less than $f(v)$,
and the set $J^+_f(v)$ consists of positive odd integers less than or equal to $f(v)+1$.
We also obtain a characterization for graphs of odd order satisfying the Tutte's condition with
respect to a function.
%$H\in \mathcal{H}_f$, where
%\[
%\mathcal{H}_f=\bigl\{H\colon V(G)\to 2^\N\ \big|\ H(v)\in \{J_f(v),\,J_f^+(v)\}\text{ for each $v\in V(G)$}\bigr\}.
%\]
%This is a new characterization on the  problem posed by Akiyama and Kano,
%as well as a problem of Cui and Kano's.
\end{abstract}

\maketitle

\section{Introduction}\label{sec:introduction}

This note connects Tutte's condition with graph factors.
Tutte's theorem states that a graph $G$ has a perfect matching if and only if
\[
o(G-S)\lle |S|\qquad\text{for any set $S\subset V(G)$},
\]
where $o(G-S)$ denotes the number of odd components of the subgraph $G-S$, 
and $V(G)$ is the vertex set of $G$. 
Let $f\colon V(G)\to \Z^+$ be a function,
where $\Z^+$ denotes the set of positive integers.
The \emph{Tutte's condition on $G$ with respect to $f$} is the condition
\[
o(G-S)\le f(S)\qquad\text{for any nonempty set $S\subset V(G)$},
\]
where $f(S)=\sum_{v\in S}f(v)$.
The Tutte's condition with respect to the constant function $f\equiv 1$ is the condition in Tutte's theorem.

A considerable large number of literatures on graph factors can be found
in Akiyama and Kano's book~\cite{AK11B}.
Let $H\colon V(G)\to 2^\N$ be a set-valued function.
A spanning subgraph $F$ of $G$ is called an \emph{$H$-factor} if $\deg_F(v)\in H(v)$.
%It is called a \emph{$(1,f)$-odd factor} if it is an $H$-factor where $H(v)=\{1,3,5,\ldots,f(v)\}$,
%and $f(v)$ is odd for each $v$.
In particular, a 1-factor is exactly a perfect matching. 
For any vertex $x$ of $G$, denote by $G^x$ the graph obtained from $G$
by adding a new vertex $x'$ together with a new edge $xx'$.
%, that is, $G^x=G+xx'$. 
A graph $G$ is said to be {\em $H$-critical}
if $G$ contains no $H$-factors and if the graph $G^x$ has an $H^x$-factor for every vertex $x$ of~$G$,
where
\[
H^x(v)=\begin{cases}
\{1\}, &\text{if $v=x'$};\\[5pt]
H(v), &\text{otherwise}.
\end{cases} 
\]

Lov\'asz~\cite{Lov69} proposed the {\em degree prescribed subgraph problem}
of determining the distance of a factor from a given integer set function.
He~\cite{Lov72} considered it with the restriction that the given set function $H$ is allowed,
i.e., that every gap of the set $H(v)$ for each vertex $v$ is at most two.
He also showed that the problem is NP-complete when the function $H$ is not allowed.
Cornu\'ejols~\cite{Cor88} provided a polynomial Edmonds-Johnson type alternating forest algorithm
for the degree prescribed subgraph problem with~$H$ allowed,
which implies a Gallai-Edmonds type structure theorem.

For convenience, we denote the set of positive odd integers by $2\N+1$, and 
\[
J_n
\=\begin{cases}
\{1,3,5,\ldots,n\},&\text{if $n$ is odd};\\[4pt]
\{1,3,5,\ldots,n-1,n\},&\text{if $n$ is even}.
\end{cases}
\]
Define $J_f(v)=J_{f(v)}$ for all vertices $v$.
%Under this notation, $J_{f}$-factors are exactly $(1,f)$-odd factors when $f(v)$ is odd for each vertex $v$.
%Amahashi~\cite{Ama85} gave a Tutte-type characterization for graphs having an odd factor.
%\begin{thm}[Amahashi, \cite{Ama85}]\label{thm:Amahashi}
%Let $n\ge 2$ be an integer.
%A general graph $G$ has a $J_{2n-1}$-factor if and only if
%\begin{equation}\label[cond]{cond:Amahashi}
%o(G-S)\le(2n-1)|S|\qquad\text{for all vertex subsets $S$}.
%\end{equation}
%\end{thm}
%
%By generalizing the constant odd integer $2n-1$ to an odd-valued function $f(v)$,
%Cui and Kano~\cite{CK88} obtained \cref{thm:CK}.
%Let $f\colon V(G)\to 2\N+1$ be a function.

\begin{thm}[Cui and Kano~\cite{CK88}]\label{thm:CK}
A connected general graph $G$ of even order satisfies the Tutte's condition 
with respect to a function $f\colon V(G)\to 2\N+1$
if and only if $G$ contains a $J_f$-factor.
\end{thm}

Extending the range of $f$ to be all positive integers,
Egawa, Kano, and Yan~\cite{EKY16} obtain \cref{thm:EKY}.

\begin{thm}[Egawa et al.~\cite{EKY16}]\label{thm:EKY}
Suppose that a connected simple graph $G$ of even order satisfies the Tutte's condition 
with respect to a function $f\colon V(G)\to\Z^+$.
Then~$G$ contains a $J_f$-factor.
\end{thm}

The particular case $f(v)\equiv 2n$ for some integer $n$ had been solved by Akiyama, Avis and Era~\cite{AAE80} 
for $n=1$ and by the present authors~\cite{LW13} for $n\ge2$.

%\DWr{\Cref{cond:f} is considered as a \emph{restricted Tutte type condition},
%for the order of $G$ satisfying it has to be even,
%which can be seen from taking $S=\emptyset$ in \cref{cond:f}.}
%By setting $S=\emptyset$, we see that \cref{prob:CK} involves graphs of even orders only.
%Taking account of graphs of odd order,
Without restricting the parity of order of $G$,
Akiyama and Kano \cite[Problem 6.14 (2)]{AK11B} proposed \cref{prob:AK}.
Denote by $2\Z^+$ the set of positive even integers.
\begin{prob}[Akiyama and Kano~\cite{AK11B}]\label{prob:AK}
Suppose that a connected simple graph~$G$ satisfies the Tutte's condition 
with respect to a function $f\colon V(G)\to 2\Z^+$.
%\begin{equation}\label[cond]{cond:AK}
%o(G-S)\lle f(S)\qquad\text{for any nonempty set $S\subset V(G)$},
%\end{equation}
Then what factor or property does $G$ have?
\end{prob}

In the next section, 
we will give a characterization of graphs satisfying the Tutte's condition with respect to a function $f$,
in terms of graph factors, without any restriction on the range of $f$,
and for graphs of any parity of order.

\section{Main Result}
In terms of graph factors, 
the authors~\cite{LW17} have characterized graphs
satisfying the Tutte's condition with respect to a function,
%depending on the parity of order of $G$,
but with the aid of 
either $2$-colorings, or $2$-edge-colorings, or $2$-end-colorings; see \cref{thm:rTutte,thm:Tutte}.
In this note, we present characterizations in terms of graph factors only;
see \cref{thm:rTutte:new,thm:Tutte:new}.

%Let $G$ be a general graph allowing both loops and parallel edges,
%with vertex set $V(G)$ and edge set $E(G)$.
%We call a function $g\colon V(G)\rightarrow \{B,R\}$ a \emph{blue-red coloring} of $G$. 
%Let
%\[
%\mathcal{C}=\{g:V(G)\rightarrow \{B,R\}\}.
%\]
%For $g\in \mathcal{C}$, let $g_B=\{v\in V(G)\ |\ g(v)=B\}$ and $g_R=\{v\in V(G)\ |\ g(v)=R\}$.

\begin{thm}[Lu and Wang~\cite{LW17}]\label{thm:rTutte}
A connected general graph $G$ of even order 
satisfies the Tutte's condition with respect to a function $f\colon V(G)\to\Z^+$
if and only if
$G$ contains an $H$-factor for any coloring $g\colon V(G)\rightarrow \{B,R\}$, where 
\[
H(v)=\begin{cases}
J_f(v),&\text{if $g(v)=R$};\\[5pt]
2\N+1,&\text{if $g(v)=B$}.
\end{cases}
\]
%\[
%o(G-S)\le f(S)\qquad\text{for all vertex subsets $S$},
%\]
%if and only if for any $g\in \mathcal{C}$, $G$ contains an $H_g$-factor, where $H_g(v)=J_f(v)$ for all $v\in g_R$ and $H_g(v)=2N+1$ for all $v\in g_B$.
\end{thm}

For any function $f\colon V(G)\to\Z^+$,
let $J_f^+(v)$ be the set of positive odd integers
that are less than or equal to $f(v)+1$. In other words,
\begin{align*}
J_f^+(v)
=\{m\in 2\N+1\colon m\le f(v)+1\}
=\begin{cases}
J_{f(v)},&\text{if $f(v)$ is odd};\\[4pt]
J_{f(v)+1},&\text{if $f(v)$ is even}.
\end{cases}
%&=\begin{cases}
%\{1,3,5,\ldots,f(v)\},&\text{if $f(v)$ is odd}\\[4pt]
%\{1,3,5,\ldots,f(v)-1,f(v)+1\},&\text{if $f(v)$ is even}.
%\end{cases}
\end{align*}
Define a set 
\[
\mathcal{H}_f=\bigl\{H\colon V(G)\to 2^\N\ \big|\ H(v)\in \{J_f(v),\,J_f^+(v)\}\text{ for each $v\in V(G)$}\bigr\}.
\]

\bigskip

\begin{thm}\label{thm:rTutte:new}
A connected general graph $G$ of even order satisfies the Tutte's condition 
with respect to a function $f\colon V(G)\to\Z^+$
%\[
%o(G-S)\le f(S)\qquad\text{for any set $S\subset V(G)$}.
%\]
if and only if $G$ contains an $H$-factor for any $H\in \mathcal{H}_f$. 
\end{thm}

\begin{proof}
Let $G$ be a connected general graph of even order, with a function $f\colon V(G)\to\Z^+$.
We shall show the necessity and sufficiency respectively.

\noindent{\bf Necessity.}
Let $H\in \mathcal{H}_f$.
Consider the function $f'\colon V(G)\to \Z^+$ defined by
 \[
f'(v)=\max_{x\in H(v)}x=\begin{cases}
f(v)+1,&\text{if $H(v)=J^+_f(v)$ and $f(v)$ is even};\\[4pt]
f(v),&\text{otherwise}.
\end{cases}
\]
From the premise, we infer immediately
\[
o(G-S)\le f(S)\le f'(S)\qquad\text{for any set $S\subset V(G)$}.
\]
Applying Theorem \ref{thm:rTutte} with the coloring $g$ such that $g^{-1}(R)=V(G)$,
one obtains that $G$ contains an $J_{f'}$-factor, i.e., an $H$-factor.

\noindent{\bf Sufficiency.}
Let $S\subset V(G)$.
Consider the function $H\in\mathcal{H}_f$ defined by
\[
H(v)=\begin{cases}
J_f(v),&\text{if $v\in S$};\\[4pt]
J^+_f(v), &\text{otherwise}.
\end{cases}
\]
From premise, the graph $G$ has an $H$-factor, say, $F$.
Let $C$ be any odd component of the subgraph $G-S$.
Then for each $v\in C$, we have $H(v)=J_f^+(v)$
and thus the degree $d_F(v)$ is odd.
By parity argument, we have $E_F(V(C),\,S)\neq \emptyset$.
Therefore, one may deduce that
\begin{align*}
o(G-S)\le \sum_{C}|E_F(V(C),S)|\le f(S).
\end{align*}
This completes the proof.
\end{proof}

We remark that \cref{thm:rTutte:new} reduces to \cref{thm:CK} if $f(V(G))\subseteq 2\N+1$.
In fact, when $f(V(G))\subseteq 2\N+1$,
we obtain $J_f=J_f^+$ and
% the set $\mathcal{H}_f$ reduces to
\[
\mathcal{H}_f=\bigl\{H\colon V(G)\to 2^\N\ |\ H(v)=J_f(v)\text{ for each $v\in V(G)$}\bigr\}=\{J_f\}.
\]
%As a consequence, one may obtain \cref{thm:CK} from \cref{thm:rTutte:new}.

\bigskip

\begin{thm}[Lu and Wang~\cite{LW17}]\label{thm:Tutte}
Let $G$ be a connected general graph.
Then $G$ satisfies the Tutte's condition with respect to a function $f\colon V(G)\to\Z^+$ 
if and only if
for any coloring $g\colon V(G)\rightarrow \{B,R\}$,
the graph $G$ either contains an $H$-factor or is $H$-critical, where 
\[
H(v)=\begin{cases}
J_f(v),&\text{if $g(v)=R$};\\[5pt]
2\N+1,&\text{if $g(v)=B$}.
\end{cases}
\]
%\[
%o(G-S)\le f(S)\qquad\text{for all non-empty vertex subsets $S$},
%\]
\end{thm}

By a proof similar to that of Theorem \ref{thm:rTutte:new}, 
one may obtain the following result.
\begin{thm}\label{thm:Tutte:new}
Let $G$ be a connected general graph of odd order.
Then $G$ satisfies the Tutte's condition with a function $f\colon V(G)\to\Z^+$ if and only if
the graph $G$ either contains an $H$-factor or is $H$-critical, for any $H\in \mathcal{H}_f$.
\end{thm}
\begin{proof}
Omitted.
\end{proof}

Combining \cref{thm:rTutte:new,thm:Tutte:new} gives an answer to  \cref{prob:AK}.

\end{document}